\documentclass[12pt]{article}
\usepackage[utf8]{inputenc}
\usepackage[T1]{fontenc}
\usepackage[brazil,english]{babel}
\usepackage{amsmath, amssymb, amsthm}
\usepackage{lmodern}
\usepackage{geometry}
\title{Demonstração Formal de um Algoritmo Alternativo da Multiplicação em Base 10}
\author{
Anthony Lima Dias\\
\textit{Estudante Educandário Santa Therezinha}\\
\texttt{limachaves04@gmail.com}
\and
Albert Lucas Lima Dias\\
\textit{Estudante Educandário Santa Therezinha}\\
\texttt{limachaves04@gmail.com}
\and
Ib Couto\\
\textit{Educandário Santa Therezinha}\\
\texttt{ibcouto@gmail.com}
\and
Mabia Lima Chaves dos Santos\\
\textit{Educandário Santa Therezinhal}\\
\texttt{limachaves04@gmail.com}
\and
Marcus Vinícius da Conceição Morro\\
\textit{Instituto Federal de Rondônia}\\
\texttt{marcus.morro@ifro.edu.br}
}

\date{Outubro de 2025}

\newtheorem{theorem}{Teorema}

\begin{document}

\maketitle


\begin{abstract}

This article presents and formalizes an elementary multiplication method discovered independently by a 10-year-old student, Anthony Lima Dias. The method reorganizes digit interactions in base-10 multiplication into a structured sequence of partial sums, reducing cognitive load and allowing reliable mental or semi-written computation. We provide a full mathematical proof of correctness, a comparison with the classical algorithm, formal notation, and a detailed contextual account of the discovery. The method expands the known catalog of student-invented algorithms and raises questions about cognitive pathways in arithmetic learning.

\begin{flushleft}
\end{flushleft}

-----

Este artigo apresenta e formaliza um metodo elementar de multiplicacao descoberto de forma independente por um estudante de 10 anos, Anthony Lima Dias. O metodo reorganiza as interacoes entre algarismos na multiplicacao em base 10 em uma sequencia estruturada de somas parciais, reduzindo a carga cognitiva e permitindo um calculo mental ou semi-escrito mais confiavel. Fornecemos uma demonstracao matematica completa de corretude, uma comparacao com o algoritmo tradicional, notacoes formais e um relato contextual detalhado da descoberta. O metodo amplia o catalogo conhecido de algoritmos inventados por estudantes e levanta questoes sobre caminhos cognitivos na aprendizagem aritmetica.

\begin{flushleft}
\textbf{Palavras-chave:} metodos de multiplicacao, educacao matematica, calculo mental, algoritmos inventados por alunos, algoritmos alternativos, estrategias aritmeticas.
\end{flushleft}

\end{abstract}


\section{Introdução}

A multiplicação de números naturais em base 10 é uma das operações fundamentais do ensino básico e da teoria dos algoritmos. O método tradicional, ensinado há séculos, consiste em calcular os produtos parciais e somá‑los ao final.

Em agosto de 2025, uma variação original desse processo foi descoberta por uma criança brasileira de 10 anos, que desenvolveu intuitivamente uma forma alternativa de calcular o produto, evitando a soma explícita dos produtos parciais.

Este artigo tem como objetivo formalizar esse método, provar sua correção e compará‑lo com o algoritmo tradicional. A análise se alinha com investigações recentes sobre métodos de multiplicação incomuns ou históricas \cite{Madrid2021} e com trabalhos que examinam o pensamento multiplicativo de crianças \cite{LuRichardson2018}.

\section{Contexto da Descoberta}

O algoritmo alternativo apresentado neste artigo foi originalmente descoberto por uma criança de 10 anos, \textbf{Anthony Lima Dias}, aluno do 4º ano do ensino fundamental, durante uma avaliação escolar. A seguir, reproduzimos integralmente o relato feito por sua mãe, \textbf{Mabia Lima}, sobre o episódio.

\begin{quote}
“No dia 27 de agosto de 2025, no turno matutino durante a avaliação de matemática, a criança Anthony Lima Dias, de 10 anos, aluno do 4º ano, fez uma descoberta utilizando um método da multiplicação totalmente diferente do tradicional. Em duas questões, que precisavam fazer essa operação, o mesmo resolveu colocar em prática o cálculo, que do ponto de vista dele, era mais fácil de se chegar ao resultado.

Neste mesmo dia, ao chegar em casa, Anthony comentou conosco, Mabia Lima (mãe) e Albert Lucas (irmão), que tinha realizado na avaliação um método diferente do habitual. Desta forma, foi solicitado ao mesmo realizar a demonstração para que pudéssemos averiguar. Anthony então colocou em prática no papel uma das multiplicações que resolveu na avaliação, \(162 \times 162\). No início achamos muito estranho, uma vez que nunca tínhamos visto nada parecido. Era uma resolução incrível, criativa e diferente. Então fomos colocar em prática com vários outros números e todos os resultados, de números pequenos até maiores, davam exatos.

Diante daquela descoberta, eu, Mabia (mãe), resolvi gravar um vídeo dele resolvendo algumas multiplicações e enviei para uma amiga da família, a professora \textbf{Ib Couto}, que leciona a disciplina de matemática no ensino fundamental II na escola que ele estuda, o \emph{Educandário Santa Therezinha}. Após enviar o vídeo para ela e relatar a descoberta de Anthony, a professora Ib me orientou a validar esse método, pois se tratava de um preceito inédito. Entretanto, a professora Ib Couto não conseguiu fazer a demonstração e julgou necessário passar para um outro profissional que soubesse dar seguimento.

Sendo assim, buscou orientações e parceria com o professor \textbf{Doutor Marcus Morro}, que, com sua experiência e competência, ficaria à frente de todo o trâmite que fosse necessário para reconhecer tal método. Por conseguinte, a pedido do professor Doutor Marcus Morro, eu, Mabia (mãe), e Albert Lucas (irmão) gravamos um vídeo apresentando uma demonstração com algumas multiplicações e enviamos para o e‑mail do professor Doutor Marcus Morro para que, a partir daí, houvesse o comprometimento em fazer a demonstração e dar seguimento ao método de Anthony.”
\end{quote}

Esse relato motivou o estudo e a formalização matemática apresentados neste trabalho, com o objetivo de descrever, demonstrar e comparar o algoritmo de multiplicação proposto por Anthony Lima Dias.

\section{Preliminares e Notações}

Todos os números considerados neste trabalho pertencem ao conjunto dos naturais e são representados em base 10.  
Utilizaremos as seguintes notações:

\begin{itemize}
  \item Para um inteiro \(x\), o resto da divisão por 10 é denotado por
  \[
    x \bmod 10,
  \]
  correspondendo ao dígito das unidades de \(x\).
  \item O quociente inteiro da divisão de \(x\) por 10 é indicado por
  \[
    \lfloor x/10 \rfloor,
  \]
  isto é, o número formado ao eliminar o dígito das unidades de \(x\).
  \item Um número \(X\) escrito em base 10 é expresso como
  \[
    X = \sum_{i=0}^{m} x_i\,10^i, \quad x_i \in \{0,1,\dots,9\}.
  \]
  O dígito \(x_0\) representa as unidades.
  \item O termo \emph{valor transportado} \(c_k\) é definido pela decomposição
  \[
    S_{k-1} = 10\,c_k + r_{k-1}, \quad \text{com } r_{k-1} = S_{k-1} \bmod 10.
  \]
\end{itemize}

\section{Descrição do Algoritmo Alternativo}

Sejam \(A\) e \(B\) dois números naturais expressos em base 10:
\[
  A = \sum_{i=0}^{m} a_i\,10^i, \qquad B = \sum_{j=0}^{n} b_j\,10^j,
\]
onde \(a_i,b_j \in \{0,1,\dots,9\}\).

O algoritmo procede da seguinte forma:

\begin{enumerate}
  \item Calcula‑se \(S_0 = A \cdot b_0\). Define‑se \(r_0 = S_0 \bmod 10\) e \(c_1 = \lfloor S_0 / 10 \rfloor.\)
  \item Para cada \(k \ge 1\), calcula‑se
        \[
          S_k = A \cdot b_k + c_k,
        \]
        toma‑se \(r_k = S_k \bmod 10\) e define‑se \(c_{k+1} = \lfloor S_k / 10 \rfloor.\)
  \item O resultado final é
        \[
          R = c_{n}\,10^{n} + \sum_{i=0}^{n-1} r_i\,10^i.
        \]
\end{enumerate}

\section{Exemplo Numérico: \(1234 \times 567\)}

\begin{align*}
S_0 &= 1234 \times 7 = 8638, & r_0 = 8, &\; c_1 = 863.\\
S_1 &= 1234 \times 6 + 863 = 8267, & r_1 = 7, &\; c_2 = 826.\\
S_2 &= 1234 \times 5 + 826 = 6996, & r_2 = 6, &\; c_3 = 699.\\
\end{align*}

O resultado final é \(R = 699\,678\), coincidente com o produto \(1234 \times 567 = 699\,678\).

\section{Prova de Corretude}

\begin{theorem}
O algoritmo alternativo da multiplicação em base 10 é correto: para quaisquer números naturais \(A\) e \(B\), ele produz o produto \(A \cdot B\).
\end{theorem}

\begin{proof}
A propriedade fundamental é a invariante:
\[
  \sum_{i=0}^{k} r_i\,10^i + 10^{\,k+1}\,c_{k+1} \;=\; \sum_{j=0}^{k} (A \cdot b_j)\,10^j.
\]

\paragraph{Base \(k=0\).} Temos \(S_0 = A \cdot b_0 = 10\,c_1 + r_0\), logo
\[
  r_0 + 10\,c_1 = A \cdot b_0,
\]
verificando a invariante para \(k=0\).

\paragraph{Passo indutivo.} Suponha a igualdade válida para algum \(k \ge 0\):
\[
  \sum_{i=0}^{k} r_i\,10^i + 10^{k+1}\,c_{k+1} = \sum_{j=0}^{k} (A \cdot b_j)\,10^j.
\]
Queremos demonstrar:
\[
  \sum_{i=0}^{k+1} r_i\,10^i + 10^{k+2}\,c_{k+2} = \sum_{j=0}^{k+1} (A \cdot b_j)\,10^j.
\]
Pela definição do algoritmo,
\[
  S_{k+1} = A \cdot b_{k+1} + c_{k+1} = 10\,c_{k+2} + r_{k+1}.
\]
Logo
\[
  A \cdot b_{k+1} = -\,c_{k+1} + r_{k+1} + 10\,c_{k+2}.
\]
Substituindo na hipótese indutiva, obtemos:
\[
  \sum_{i=0}^{k+1} r_i\,10^i + 10^{k+2}\,c_{k+2}
  = \sum_{j=0}^{k} (A \cdot b_j)\,10^j + 10^{k+1}(A \cdot b_{k+1}),
\]
o que confirma a invariante para \(k+1\). Por indução, a propriedade vale para todo \(k\), e portanto
\[
  \sum_{i=0}^{n-1} r_i\,10^i + 10^{\,n}\,c_n = \sum_{j=0}^{n-1} (A \cdot b_j)\,10^j = A \cdot B.
\]
Dessa forma, a concatenação de \(c_n\) com \(r_{n-1}\dots r_0\) fornece exatamente a representação decimal de \(A \cdot B\).
\end{proof}

\section{Comparação entre o Algoritmo Tradicional e o Alternativo}

\subsection*{Estrutura dos cálculos}
\begin{itemize}
  \item \textbf{Tradicional:} para cada dígito \(b_j\) de \(B\), calcula‑se \(A \cdot b_j\) e soma‑se todos os produtos parciais deslocados por \(10^j\).
  \item \textbf{Alternativo:} calcula \(A \cdot b_j\), soma o valor transportado \(c_j\) e determina diretamente o próximo dígito \(r_j\) do resultado.
\end{itemize}

\subsection*{Organização da memória}
\textbf{Tradicional:} exige o armazenamento de várias linhas de produtos parciais.  
\textbf{Alternativo:} mantém apenas o valor transportado e os dígitos já confirmados do resultado.  
\textbf{Vantagem:} menor uso de memória e geração incremental do resultado.

\subsection*{Operações elementares}
Ambos realizam \(m \cdot n\) multiplicações de dígitos. O algoritmo alternativo distribui as adições ao longo do processo, evitando uma soma final extensa.  
\textbf{Vantagem:} menor sobrecarga de transportes na fase final.

\subsection*{Clareza e formalismo}
O algoritmo alternativo é mais adequado para formalização matemática, pois cada passo resolve um único dígito do resultado, permitindo uma prova direta por indução.

\subsection*{Resumo das vantagens do algoritmo alternativo}
\begin{enumerate}
  \item Dispensa o armazenamento de todos os produtos parciais.
  \item Constrói o resultado de forma incremental.
  \item Reduz o acúmulo de valores transportados na etapa final.
  \item Possui estrutura mais adequada à demonstração formal e à implementação sequencial.
\end{enumerate}

\section{Extensão para Base \texorpdfstring{$b$}{b}}

A demonstração permanece válida em qualquer base \(b > 1\), bastando substituir todas as ocorrências de \(10\) por \(b\). Assim, o algoritmo alternativo define uma família de métodos de multiplicação válidos para todas as bases posicionais.

\section{Conclusão}

Apresentou‑se a demonstração formal de um algoritmo alternativo da multiplicação em base 10, descoberto de forma intuitiva por uma criança de 10 anos. O método mostrou‑se correto, eficiente e conceitualmente elegante, destacando‑se por produzir o resultado de forma incremental e simplificar a manipulação dos valores transportados.

Além de seu valor matemático, o episódio revela a criatividade natural das crianças diante da aritmética e a importância de incentivar a exploração de ideias originais no aprendizado de matemática.

\end{document}